\documentclass{amsart}
\usepackage{amsaddr}
\usepackage{amsfonts,amssymb,amsmath,amsthm}
\usepackage{url}
\usepackage{enumerate}
\usepackage{hyperref}
\usepackage{graphicx}
\usepackage{siunitx}
\usepackage{textcase}
\usepackage{float}
\restylefloat{figure}
\newtheorem{theorem}{Theorem}[]

\newtheorem{problem}{Problem}
\newtheorem{conjecture}{Conjecture}
\newtheorem{proposition}{Proposition}

\author{Orgil-Erdene Erdenebaatar \MakeLowercase{\and} Uuganbaatar Ninjbat}
\address{
Department of Mathematics \\
The National University of Mongolia, Ikh Surguuliin gudamj-1, 14201\\
Ulaanbaatar, Mongolia.}
\email{e.orgio0909@gmail.com\and uuganbaatar@num.edu.mn}
\thanks{To appear in Elemente der Mathematik}

\begin{document}
\title[The smallest convex $k$-gon containing $n$ congruent disks]{The smallest convex $k$-gon containing $n$ congruent disks}
\begin{abstract}
Consider the problem of finding the smallest area convex $k$-gon containing $n\in \mathbb{N}$ congruent disks without an overlap. By using Wegner inequality in sphere packing theory we give a lower bound for the area of such polygons. For several cases where this bound is tight we construct corresponding optimal polygons. We also discuss its solution for some cases where this bound is not tight, e.g. $n=2$ and $k$ is odd, and $n=3, k=4$. On the way to prove our results we prove a result on geometric invariants between two polygons whose sides are pairwise parallel, and give a new characterisation for the trisectrix of Maclaurin.
\end{abstract}
\maketitle

\section{Introduction}
\label{intro}
Most finite sphere packing problems fall into two types \cite{handbook}.
\begin{itemize}
\item \textit{Free packing}: Locate a finite set of congruent balls in the space so that the volume of their convex hull is minimal. In the two dimensional space, notable results on this problem are the Thue-Groemer and Oler inequalities (see Chap. 4.3 in \cite{finite-pack}, \cite{folkman}) and the Wegner inequality (see Theorem \ref{wegner}). In higher dimensions, L.F. T\'{o}th's sausage conjecture is a partially solved major open problem \cite{sausage}.
\item \textit{Bin packing}: Locate a finite set of congruent balls in the smallest volume container of a specific kind. In the two dimensional space, the container is usually a circle \cite{fodor}, an equilateral triangle \cite{pack-tri} or a square \cite{pack-sq}. In such cases, the smallest containers and the corresponding optimal packings are known when the number of disks is not so big, e.g. up to 20 \cite{graham}.
\end{itemize}

We study the following finite packing problem which contains elements of both types. 
\begin{problem}
\label{main-pr}
For $k\geq 3$ find the smallest area convex $k$-gon containing $n\in \mathbb{N}$ unit radius ('unit' in brief) disks without an overlap.
\end{problem}
The solution to this problem for $n=1$ is given by the following well-known result (see Chap. 2 in \cite{andreescu}). 
\begin{theorem}
\label{1-thm}
When $n=1$ the regular $k$-gon circumscribing a unit disk is the only solution to Problem \ref{main-pr} for all $k\geq 3$. 
\end{theorem}
In Theorem \ref{n-thm} of this paper we give an extension of this result as an inequality bounding area of the containing polygon from below. This inequality is tight in many cases including  
\begin{itemize}
\item $n=1$ and $k\geq 3$,
\item $n=2$ and $k=2k'$ for $k'\geq 2$, 
\item $n\in \{3, 6\}$ and $k=3k'$ and 
\item $n$ is a centered hexagonal number and $k=6k'$.
\end{itemize}
The solution of Problem \ref{main-pr} for these tight cases is obtained in Theorem \ref{6k-thm}. Then we discuss its solution for other cases where this bound is not binding; for $n=2, k=2k'+1$ in the remark following Theorem \ref{6k-thm}, and for $n=3, k=4$ in Theorem \ref{non-wegner}. The latter case is essential as it demonstrates the possibility of disks being packed non-efficiently inside the minimal polygon. 

Along the way to prove our main results, we prove two intermediate results which are interesting on their own. The first one gives geometric invariants between two polygons whose sides are pairwise parallel; see Proposition \ref{parallel}. The second one gives a simple geometric characterisation for a well known curve, the trisectrix of Maclaurin;  see Proposition \ref{maclaurin}. In the final section we discuss some open problems.

\section{Preliminaries}
In addition to the usual ones, we use the following definitions. A {\em region} is a subset of the plane with finite area and when $X$ is a region, $\|X\|$ denotes its area. The line segment connecting points $A, B$ is denoted as $AB$, $|AB|$ is its length and $(AB)$ is its interior. Let a finite set of unit disks be located in $\mathbb{R}^{2}$ without an overlap, i.e. each pair has disjoint interior. Their {\em joint tangent} is a line which is tangent to at least two of the disks and supports their convex hull. Each joint tangent bounds a half plane which contains the disks. Intersection of these half planes is called as {\em tangent polygon} of the disks (see Fig. \ref{basic} $(a)$, $(b)$). Clearly, every tangent polygon is convex and it is a convex polygon as long as the centres of the disks are not all collinear.
\begin{figure}[h]
\centering
\includegraphics[height=2.0 in, keepaspectratio = true]{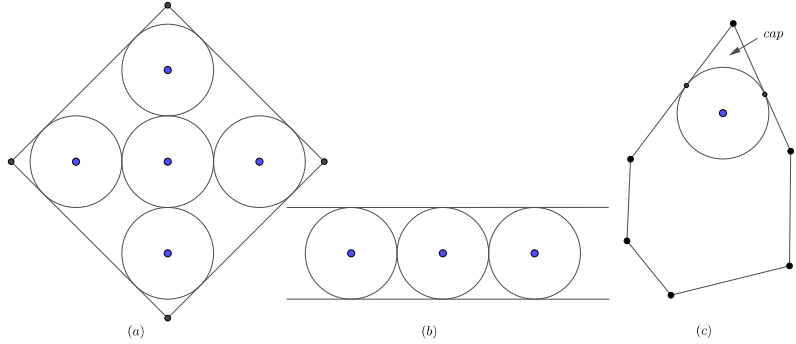}
\caption{Tangent polygons (a quadrilateral and an infinite strip) and a cap.}
\label{basic}
\end{figure}

The regular $k$-gon circumscribing a unit disk is called as {\em unit $k$-gon}. Thus, unit triangle has sides of length $2\sqrt{3}$, while unit square has sides of length $2$. {\em $\frac{1}{m}'$th of a regular $mk$-gon} is a polygon obtained after cutting the original polygon by two apothems intersecting at $\frac{2\pi}{m}$ angle; see Fig. \ref{cut}.
\begin{figure}[h]
\centering
\includegraphics[height=1.5 in, keepaspectratio = true]{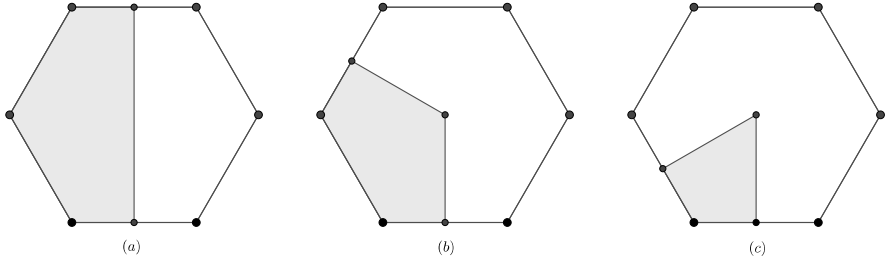}
\caption{Half, one-third, one-sixth of a regular hexagon.}
\label{cut}
\end{figure}

Let $C\subset \mathbb{R}^{2}$ be a convex disk and $P$ be a convex polygon containing $C$. Using a terminology in Pach and Agarwal \cite{pach}, a {\em cap of $P$ w.r.t. $C$} is the region enclosed by the boundary of $C$ and two consecutive sides of $P$ which are tangent to $C$; see Fig. \ref{basic} $(c)$. $\|Cap_{C}(P)\|$ denotes the sum of areas of all caps of $P$ w.r.t. $C$. 

A set of $n$ unit disks constitute a {\em Groemer packing} if each pair has disjoint interior and the convex hull of their centres is either a line segment of length $2(n-1)$ or can be triangulated into equilateral triangles of edge length two using the $n$ centres as vertices \cite{groemer}. If in addition, perimeter of the hull is $2\left \lceil{\sqrt{12n-3}-3}\right \rceil$, where $\left \lceil{x}\right \rceil=\min\{z\in \mathbb{Z}: z\geq x\}$, then the Groemer packing is a {\em Wegner packing} \cite{boroczky-ruzsa}; see Fig. \ref{wegner-fig}. By using these geometric properties, it is easy to show that the convex hull of the centres of the disks in a Wegner packing has at most six sides; see Chap. 4.3 in \cite{finite-pack}. $n\in\mathbb{N}$ is called {\em exceptional} if there is no Wegner packing of $n$ unit disks. The smallest exceptional number is $121$ and they constitute less than 5\% of all $\mathbb{N}$ \cite{boroczky-ruzsa}. The following result known as Wegner inequality; see Chap. 4.3 in \cite{finite-pack}, \cite{boroczky-ruzsa}.
\begin{theorem}(Wegner inequality)
\label{wegner}
If $H$ is the convex hull of $n\in\mathbb{N}$ non-overlapping unit disks then \[\|H\|\geq \sqrt{12}\cdot(n-1)+(2-\sqrt{3})\cdot \left \lceil{\sqrt{12n-3}-3}\right \rceil+\pi.\]
Equality holds if and only if the disks are packed in a Wegner packing.
\end{theorem}
\begin{figure}[h]
\centering
\includegraphics[height=1.5 in, keepaspectratio = true]{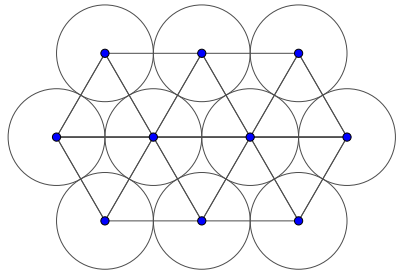}
\caption{Wegner packing of $n=10$ disks.}
\label{wegner-fig}
\end{figure}

The following result is a reliable tool when one studies the smallest circumscribing polygons of convex figures \cite{chakerian, zalgaller}.
\begin{theorem}
\label{area}
Let $C\subset \mathbb{R}^{2}$ be a convex disk and $P$ be the smallest area convex polygon containing it. Then, midpoints of sides of $P$ lie on the boundary of $C$.
\end{theorem}
If an internal angle at a vertex of a polygon is greater than $\pi$, then the vertex is {\em reflex}. If $P=A_{1}...A_{n}$ and $Q=B_{1}...B_{n}$ are two simple polygons with the same orientation and $A_{i}A_{i+1}\parallel B_{i}B_{i+1}$ for all $1\leq i\leq n$, then they are called {\em parallel polygons}. Finally, the {\em Maclaurin trisectrix} is a cubic plane curve defined as the locus of the point of intersection of two lines, each rotating at a uniform rate about separate points, so that the ratio of the rates of rotation is 1:3 and the lines initially coincide with the line passing the two points. Its polar equation is $r=a\sec \frac{\theta}{3}$ and Cartesian equation is $y^{2}=\frac{x^{2}(x+3a)}{a-x}$ \cite{encyclo}.

\section{The main results}
\label{main}
We shall prove two intermediate results. 
\begin{proposition}
\label{parallel}
Let $P$ and $Q$ be two simple parallel polygons. Then 
\begin{itemize}
\item[(a)] They have the same number of reflex vertices, and 
\item[(b)] If one is convex so is the other and their corresponding internal angles are equal.
\end{itemize}
\end{proposition}
\begin{proof}
Let $P=A_{1}...A_{n}$ and $Q=B_{1}...B_{n}$ and we denote their internal angles as $\measuredangle A_{i}=\alpha_{i}$, $\measuredangle B_{i}=\beta_{i}$ for $1\leq i\leq n$. Let $I=\{i\in\mathbb{N}: 1\leq i\leq n, \alpha_{i}=\beta_{i}\}$, $I_{1}=\{i\in\mathbb{N}: 1\leq i\leq n, \alpha_{i}<\beta_{i}\}$ and $I_{2}=\{i\in\mathbb{N}: 1\leq i\leq n, \alpha_{i}>\beta_{i}\}$. Since $A_{i}A_{i+1}\parallel B_{i}B_{i+1}$ for $1\leq i\leq n$, $i\in I_{1}$ implies $\alpha_{i}+\pi=\beta_{i}$ and $i\in I_{2}$ implies $\alpha_{i}=\beta_{i}+\pi$.

Note that \[\sum_{i\in I}\alpha_{i}+\sum_{i\in I_{1}}\alpha_{i}+\sum_{i\in I_{2}}\alpha_{i}=\sum_{i\in I}\beta_{i}+\sum_{i\in I_{1}}\beta_{i}+\sum_{i\in I_{2}}\beta_{i}=(n-2)\pi.\] Since $\sum_{i\in I}\alpha_{i}=\sum_{i\in I}\beta_{i}$, $\sum_{i\in I_{1}}(\alpha_{i}+\pi)=\sum_{i\in I_{1}}\beta_{i}$ and $\sum_{i\in I_{2}}\alpha_{i}=\sum_{i\in I_{1}}(\beta_{i}+\pi)$, we get \[\sum_{i\in I_{1}}\alpha_{i}+\sum_{i\in I_{2}}\beta_{i}+|I_{2}|\pi=\sum_{i\in I_{1}}\alpha_{i}+|I_{1}|\pi+\sum_{i\in I_{2}}\beta_{i}\]
which simplifies to $|I_{1}|=|I_{2}|$. This proves Proposition \ref{parallel} (a) after noting that $|I_{1}|$ is the number of reflex vertices in $Q$ whose corresponding vertex in $P$ is normal, i.e. non-reflex, and $|I_{2}|$ is the number of reflex vertices in $P$ whose corresponding vertex in $Q$ is normal.

Assume $P$ is convex. Then, by part (a) both polygons have $0$ reflex vertices. Hence, $Q$ is convex. Since $n=|I|+|I_{1}|+|I_{2}|$ and $|I_{1}|=|I_{2}|=0$, we have $n=|I|$, i.e. $\alpha_{i}=\beta_{i}$ for all $1\leq i\leq n$. 
\end{proof}
\noindent{\em Remarks}: There seem to be a slight confusion in the computational geometry literature regarding to geometric invariants between parallel polygons. For example, on p.2 \cite{guibas} it is (mistakenly) claimed that ``two polygons are parallel iff they have the same sequence of angles". The above result clarifies the situation.

Let $P$ be a convex $k$-gon containing $n$ unit disks without an overlap such that each side of $P$ is tangent to at least one of the disks. Let us pick one of the disks and for each side of $P$ there are two tangents to the disk which are parallel to it. Choose the one which is closer to the side and the polygon whose sides are contained in these tangents is called as {\em shrink of} $P$ for the picked disk; see Fig. \ref{shrink}. By construction, $P$ and its shrink are parallel and by Proposition \ref{parallel} (b) they have the same internal angles.

Our second intermediate result is as follows. 
\begin{proposition}
\label{maclaurin}
Let $\omega$ be a circle with centre $O$ and radius $r_{\omega}$, $l$ be a line tangent to $\omega$ at $T$ and $X$ be an arbitrary point on $l$. Let $m$ be the other tangent from $X$ to $\omega$, $R=m\cap\omega$ and $X'$ be reflection of $X$ on $m$ w.r.t. $R$. Then $t(\omega, l)$, the locus of $X'$, is a Maclaurin trisectrix. Conversely, if $t(\omega, l)$ is a Maclaurin trisectrix then there exist circle $\omega$ and line $l$ which generates it as described.    
\end{proposition}
\begin{proof}
Let us introduce a polar coordinate system with pole at $O$ and axis on $TO$-ray. The angular coordinates are measured in the counterclockwise direction; see Fig. \ref{polar}.
\begin{figure}[h]
\centering
\includegraphics[height=2.2 in, keepaspectratio = true]{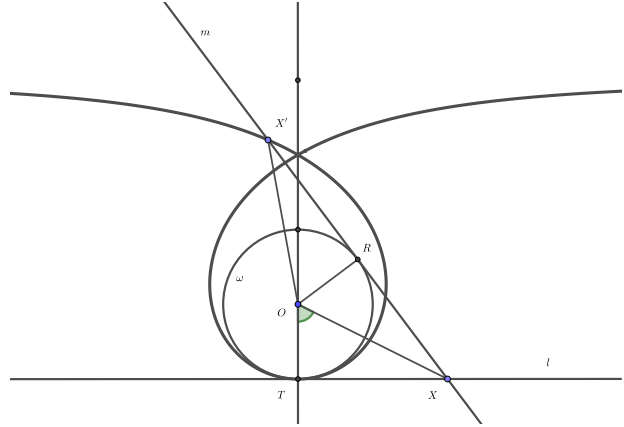}
\caption{A characterisation of Maclaurin trisectrix.}
\label{polar}
\end{figure}
Let the azimuth of $X$ be $\phi\in [0, \frac{\pi}{2})\cup (\frac{3\pi}{2}, 2\pi]$. If $\phi\in [0, \frac{\pi}{2})$, then $|OX|=r_{\omega}\sec \phi$. By construction $\triangle OX'X$ is isosceles with $|OX'|=|OX|$. Since $\triangle OTX=\triangle ORX=\triangle ORX'$ we have $\measuredangle TOX'=3\phi$. Thus, $X'$ has polar coordinates $(r_{\omega}\sec \phi, 3\phi)$ which gives the polar equation $r=r_{\omega}\sec \frac{\theta}{3}$. If $\phi\in (\frac{3\pi}{2}, 2\pi]$ we get the same equation after replacing $\phi$ in our analysis by $\phi'=2\pi-\phi$. Note that $t(\omega, l)$ has an asymptote which is perpendicular to the polar axis and passes over the point $(2\pi, 3r_{\omega})$. 

Conversely, suppose $t(\omega, l)$ is a curve with polar equation $r=a\sec \frac{\theta}{3}$. Then $\omega$ is chosen as the circle with centre at the pole and radius $a$, $l$ is tangent to $\omega$ at point $T(a, 0)$. We can repeat the above argument to show that this configuration generates $t(\omega, l)$.  
\end{proof}
\noindent{\em Remarks}: Another derivation of this curve and some motivations for finding alternative derivations of classical curves are given in \cite{loci}.

Let us now prove our first main result. 
\begin{theorem}
\label{n-thm}
If $P$ is a convex $k$-gon containing $n\in\mathbb{N}$ non-overlapping unit disks then \[\|P\|\geq \sqrt{12}\cdot(n-1)+(2-\sqrt{3})\cdot\left \lceil{\sqrt{12n-3}-3}\right \rceil+k\cdot\tan \frac{\pi}{k}.\]
Equality holds if and only if the disks are located in a Wegner packing, $P$ is equiangular, each side of $P$ is tangent to at least one of the disks and each cap of $P$ w.r.t. the convex hull of the disks is a cap w.r.t. a unit disk.
\end{theorem}
\begin{proof}
By Theorem \ref{area} we may assume that each side of $P$ is tangent to the convex hull of the disks, which we denote as $H$. This is equivalent to assume that each side is tangent to at least one of the disks. Since $\|P\|=\|H\|+\|P\setminus H\|$, by Theorem \ref{wegner} it suffices to prove that \[\|P\setminus H\|\geq k\cdot\tan \frac{\pi}{k}-\pi \tag{1}\] where the right hand side is the sum of cap areas of a unit $k$-gon w.r.t. the circumscribed unit disk, while the left hand side is $\|Cap_{H}(P)\|$.

Let $\omega$ be one of the disks and $P'$ be the shrink of $P$ for $\omega$; see Fig. \ref{shrink} (a).
\begin{figure}[h]
\centering
\includegraphics[height=2.0 in, keepaspectratio = true]{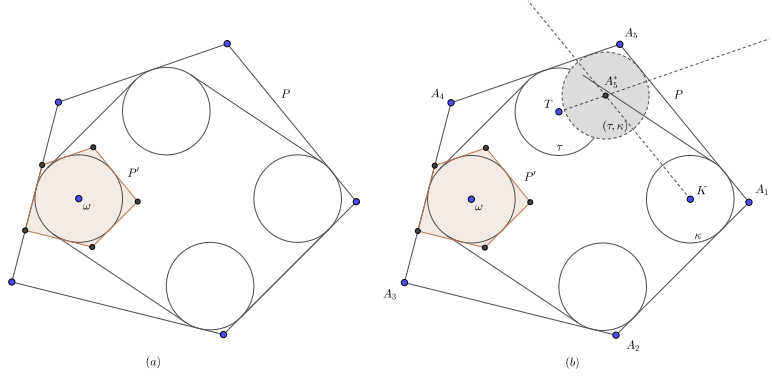}
\caption{Shrinking $P$ to $P'$ in (a); adding an auxiliary disk in (b).}
\label{shrink}
\end{figure}
Since $P'$ is a circumscribing convex $k$-gon of $\omega$, by Theorem \ref{1-thm} \[\|Cap_{\omega}(P')\|\geq k\cdot\tan \frac{\pi}{k}-\pi \tag{2}\] and equality holds iff $P'$ is regular. 

We claim that each cap of $P'$ w.r.t. $\omega$ is smaller (in area) than the corresponding cap of $P$ w.r.t. $H$. To see this, let $A_{1}, ..., A_{k}$ be the vertices of $P$ and $\mathbb{S}$ be the set of $n$ disks that it contains. Choose one of the vertices, $A_{i}$. If $A_{i-1}A_{i}$ and $A_{i}A_{i+1}$ are tangent to the same unit disk in $\mathbb{S}$, then the cap of $P$ w.r.t. $H$ with a vertex at $A_{i}$ is a cap w.r.t. a unit disk. Since shrinking preserves the internal angle at $A_{i}$, the corresponding cap of $P'$ is a translation of this cap; see Fig. \ref{shrink} (b). Thus, they are congruent.

If otherwise, $A_{i-1}A_{i}$ and $A_{i}A_{i+1}$ are tangent to two different unit disks in $\mathbb{S}$. Let they be $\tau$ with centre $T$ and $\kappa$ with centre $K$, respectively. Draw a line passing through $T$ which is parallel to $A_{i-1}A_{i}$ and another line passing through $K$ which is parallel to $A_{i}A_{i+1}$ as in Fig. \ref{shrink} (b). Let $A^{\ast}_{i}$ be their intersection. Note that this point is well defined, on the bisector of $\angle A_{i}$ and because of convexity it is inside $P$ and closer to $A_{i}$ than both $T$ and $K$. Let $(\tau,\kappa)$ be the unit disk centred at $A^{\ast}_{i}$. By construction $(\tau,\kappa)$ is tangent to  $A_{i-1}A_{i}$ and $A_{i}A_{i+1}$. If $H'$ is the convex hull of $\mathbb{S}\cup \{(\tau,\kappa)\}$, then the cap of $P$ w.r.t. $H'$ with a vertex at $A_{i}$ is the same as the corresponding cap of $P'$ w.r.t. $\omega$ by the above argument. But the former cap is smaller than the cap of $P$ w.r.t. $H$ with a vertex at $A_{i}$ by the amount \[\|H'\|-\|H\|=\|\triangle A^{\ast}_{i}TK\|+|TA^{\ast}_{i}|+|A^{\ast}_{i}K|-|TK|>0\] (recall the triangle inequality). This proves our claim which implies that \[\|Cap_{H}(P)\|\geq \|Cap_{\omega}(P')\|\tag{3}\]
and equality holds iff each cap of $P$ w.r.t. $H$ is a cap w.r.t. a unit disk. Inequalities $(2)$ and $(3)$ imply $(1)$. 

From the proof above it should be clear that equality holds iff 
\begin{itemize}
\item Each side of $P$ is tangent to at least one of the disks, and
\item $\|H\|=\sqrt{12}\cdot(n-1)+(2-\sqrt{3})\cdot \left \lceil{\sqrt{12n-3}-3}\right \rceil+\pi$, which happens iff the disks constitute a Wegner packing by Theorem \ref{wegner}, and
\item Each cap of $P$ w.r.t. the convex hull of the disks is a cap w.r.t. a unit disk, and
\item $P'$ is regular, which happens when $P$ is equiangular (recall that $P'$ and $P$ have the same internal angles).
\end{itemize} 
\end{proof}
\noindent{\em Remarks}: In above we used Theorem \ref{wegner} to prove Theorem \ref{n-thm}. One should notice that the reverse implication is also possible and very much the same. 

The following result shows that the above inequality is tight in many cases. 
\begin{theorem}
\label{6k-thm}
If
\begin{itemize}
\item[(a)] $n=2$ and $k=2k'$ with $k'\geq 2$, or
\item[(b)] $n\in \{3, 6\}$ and $k=3k'$, or
\item[(c)] $n\in\mathbb{N}\setminus\{2\}$ is not exceptional and $k=6k'$
\end{itemize}
then the inequality in Theorem \ref{n-thm} is tight and the solution of Problem \ref{main-pr} can be constructed. In particular, when $n=3m(m-1)+1$ and $k=6$ the solution is the regular hexagon with sides of length $2(m-1)+\frac{2}{\sqrt{3}}$. 
\end{theorem}
\begin{proof}
Let us prove (a). Two unit disks are Wegner packed iff they are tangent. Let $O_{1}, O_{2}$ be the centres of the disks and let us construct the solution of Problem \ref{main-pr} as follows. First, draw the tangent polygon of the disks; recall that in this case it is an infinite strip. Cut out a rectangle whose two opposite sides are contained in the joint tangents of the disks, and the other two sides pass through $O_{1}$ and $O_{2}$. For each of these latter two sides, take half of unit $2k'$-gon and paste (i.e. glue) it over its side of length $2$ with the rectangle. The resulting $2k'$-gon contains the two disks and satisfies all the equality requirements in Theorem \ref{n-thm}; see Fig. \ref{2-opt}.
\begin{figure}[h]
\centering
\includegraphics[height=1.3 in, keepaspectratio = true]{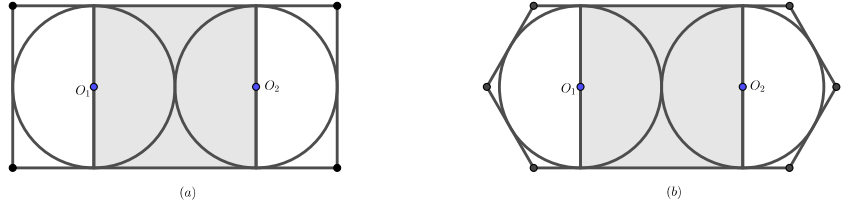}
\caption{Construction of optimal polygons for $n=2$ and $k=4$ in (a), $k=6$ in (b).}
\label{2-opt}
\end{figure}

Let us prove (b). Let $n=3$ and it is easy to see that centres of three Wegner packed disks constitute vertices of an equilateral triangle with sides of length $2$ as in Fig. \ref{tri-fig} $(a)$. Then, their tangent polygon $T_{1}$ satisfies all the equality conditions in Theorem \ref{n-thm}. Thus, $T_{1}$ is the solution when $k'=1$. For $k'=2$, we can cut from $T_{1}$ three equilateral triangles, each has a vertex common with $T_{1}$ and a side tangent to convex hull of the circles. The remaining hexagon which we denote by $T_{2}$ satisfies the necessary conditions; see Fig. \ref{tri-fig} $(b)$. In general, $T_{k'}$ is constructed as follows. Remove all caps of $T_{1}$ w.r.t. the convex hull of the disks and replace each by the union of caps of one-third of unit $3k'$-gon w.r.t. the circumscribed unit disk. It is easy to check that this construction is well defined and satisfies the necessary conditions. 

If $n=6$, a similar argument proves the statement after noting that centres of six Wegner packed disks constitute vertices of an equilateral triangle with sides of length $4$ as in Fig. \ref{tri-fig} (c).
\begin{figure}[h]
\centering
\includegraphics[height=1.6 in, keepaspectratio = true]{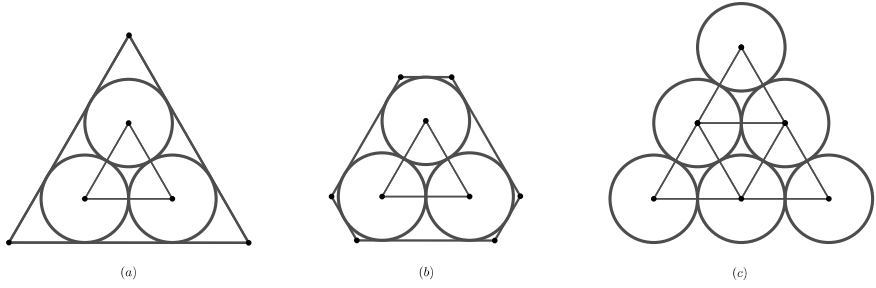}
\caption{Wegner packing and optimal polygons when $n=3, 6$.}
\label{tri-fig}
\end{figure}
 
Let us prove (c). Since $n$ is not exceptional, there is a Wegner packing of $n$ disks. Because of the perimeter condition, linear packing of $n$ disks is never a Wegner packing for $n\neq 2$. Thus, the convex hull of the centres of Wegner packed $n\neq 2$ disks which we denote as $T_{c}$ has at least three sides. Let $T$ be the tangential polygon of the packed disks. Since $T_{c}$ has at least three and at most six sides and is triangulated into equilateral triangles, and $T_{c}$ and $T$ are parallel, $T$ has three to six sides and its internal angles are either $\frac{\pi}{3}$ or $\frac{2\pi}{3}$ (recall Proposition \ref{parallel} (b)). This gives us the following possibilities:
\begin{itemize}
\item $T$ is an equilateral triangle, or
\item $T$ is a quadrilateral with two internal angles of $\frac{2\pi}{3}$ and two of $\frac{\pi}{3}$, or
\item $T$ is a pentagon with four internal angles of $\frac{2\pi}{3}$ and one of $\frac{\pi}{3}$, or
\item $T$ is a hexagon with internal angles of $\frac{2\pi}{3}$.
\end{itemize}
In each case, take a vertex with angle $\frac{\pi}{3}$ and cut out an equilateral triangle from $T$ which shares this vertex and the side of it which does not contain this vertex is tangent to the disk closest to the vertex. Let the resulting polygon be $T_{1}$. Then, by construction, $T_{1}$ is a convex hexagon containing all the disks, each of its internal angle is $\frac{2\pi}{3}$, and each side is tangent to at least one of the disks and each cap is a cap w.r.t. a unit disk. Thus, $T_{1}$ satisfies all equality conditions in Theorem \ref{n-thm} and is a solution to Problem \ref{main-pr} when $n$ is not exceptional and $k=6$.

In general, $T_{k'}$ is constructed as follows. Remove all six caps of $T_{1}$ w.r.t. the convex hull of the disks and replace each by the union of caps of a one sixth of the unit $6k'$-gon w.r.t. the unit disk which it circumscribes. It is easy to check that this construction is well defined and satisfies the equality conditions. 

If $n=3m(m-1)+1$, i.e. the centred hexagonal number, it is easy to verify that the Wegner packing of $n$ disks is so that the convex hull of their centres is a regular hexagon with sides of length $2m$. Then, their tangent polygon is the regular hexagon of sides of length $2(m-1)+\frac{2}{\sqrt{3}}$ and satisfies all the equality conditions in Theorem \ref{n-thm}. Thus, the tangent hexagon is the solution.
\end{proof}
\noindent{\em Remarks}: With some effort one can show that a construction similar to that in Theorem \ref{6k-thm} (a) works for $n=2$ and $k=2k'+1$. This time, we need to paste half of unit $2k'$-gon and half of unit $2(k'+1)$-gon to the central rectangle. Moreover, this reasoning can be applied to finding the smallest convex $k$-gon contianing $n$ linearly packed disks.

So far our solutions for Problem \ref{main-pr} relied on the cases where the disks are Wegner packed, i.e. efficiently packed. The following result shows that this is not always the case.
\begin{theorem}
\label{non-wegner}
Let $MNKL$ be the smallest area convex quadrilateral containing three efficiently packed unit disks and $P^{\ast}$ be the $2\times 6$ rectangle in which the disks are packed linearly. Then $\|MNKL\|>\|P^{\ast}\|$.
\end{theorem}
\begin{proof}
Let $\omega_{i}$, $i=1,2,3$ be the disks and $O_{i}$ be their centres. We know that $|O_{1}O_{2}|=|O_{2}O_{3}|=|O_{3}O_{1}|=2$. Let $\triangle ABC$ be the tangent polygon of the disks such that $A,B$ are on the joint tangent of $\omega_{1}, \omega_{2}$, $B, C$ are on the joint tangent of $\omega_{2}, \omega_{3}$ and $C, A$ are on the joint tangent of $\omega_{3}, \omega_{1}$. Let $\omega_{1}\cap AB=T_{1}$, $\omega_{2}\cap AB=T_{2}$, $\omega_{2}\cap BC=T_{3}$, $\omega_{3}\cap BC=T_{4}$, $\omega_{3}\cap AC=T_{5}$ and $\omega_{1}\cap AC=T_{6}$. Further let $\triangle AA'A''$ be such that $A'\in AB$, $A''\in AC$ and $\omega_{1}$ is an excircle tangent to $A'A''$ on its midpoint. Let $\triangle BB'B''$ and $\triangle CC'C''$ be defined analogously; see Fig. \ref{tangent-tri}. 

Let us draw the section of Maclaurin trisectrix for $\omega_{1}$ and $AB$-line ranging between $A''$ and the reflection $A$ w.r.t. $T_{6}$. We call this curve as the trisectrix for $(\omega_{1}, AA')$. Draw similarly trisectrices for $(\omega_{1}, AA'')$, $(\omega_{2}, BB')$, $(\omega_{2}, BB'')$, $(\omega_{3}, CC')$ and $(\omega_{3}, CC'')$. Let $I_{1}, I_{2}, I_{3}$ denote pairwise intersections of these six curves. Let $R_{1}\in(AA')$ be such that $R_{1}I_{3}$ is tangent to $\omega_{1}$ by its midpoint. We define other points $R_{i}$, $2\leq i\leq 6$ analogously. Let $r_{ij}$ with $i,j\in\{1,2,3\}$ and $i<j$ denote the radical axis of $\omega_{i}$ and $\omega_{j}$ and $O$ be their common intersection. Because of symmetry in our configuration, $I_{1}\in r_{12}$, $I_{2}\in r_{23}$, $I_{3}\in r_{13}$. Let $r_{12}\cap AB=P_{1}$ and $r_{23}\cap BC=P_{2}$.
\begin{figure}[h]
\centering
\includegraphics[height=3 in, keepaspectratio = true]{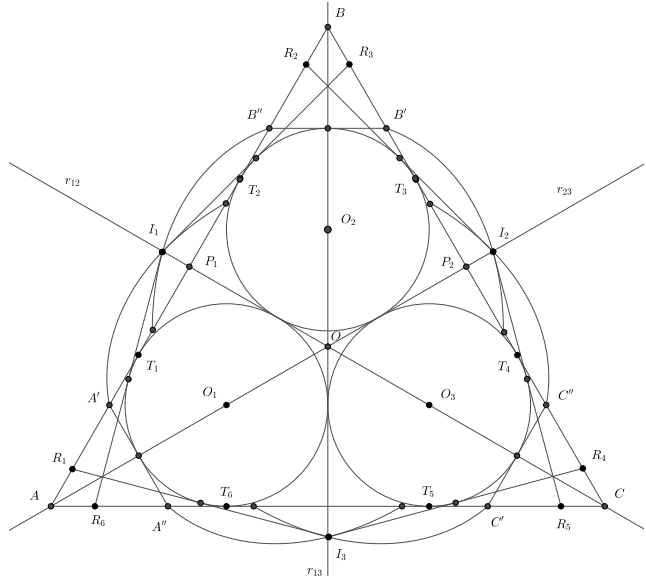}
\caption{Tangent polygon of the three disks.}
\label{tangent-tri}
\end{figure}

\noindent{\em\bf Claim:} $MNKL\in \{A'BCA'', B'CAB'', C'ABC'', I_{1}R_{3}CR_{6}, I_{2}R_{5}AR_{2}, I_{3}R_{1}BR_{4}\}$.

We assume that $MNKL$ is clockwise oriented. Let $H$ be the convex hull of $\omega_{1}, \omega_{2}, \omega_{3}$. Then \[\|H\|=\pi+6+\sqrt{3}. \tag{4}\] By Theorem \ref{area} we know that each side of $MNKL$ is tangent to $H$ on its midpoint. This implies that each side is tangent to at least one of the disks. Since there are four sides and three disks, one disk must be tangent to at least two consecutive sides. Let these be $MN$ and $ML$. This implies $M\in\triangle ABC$. Let us first assume that $M$ is located in the $A$-vertex cap of $\triangle ABC$, i.e. $MN$, $ML$ are tangent to $\omega_{1}$. We shall analyse the following cases. 
\begin{itemize}
\item[(a)] Assume $M=A'$. Then (by Theorem \ref{area}) $L=A''$, $N$ is on $A'B$-ray while $K$ is on $A''C$-ray. Since $NK$ must be tangent to $H$, this implies $N=B$ and $K=C$. So we end up with $A'BCA''$. The case of $M=A''$ is treated analogously.
\item[(b)] Assume $M$ is located in the $A'$-vertex cap of $A'BCA''$ w.r.t. $\omega_{1}$. Then $L$ is located either in the same cap or in $\triangle AA'A''$ while $N$ is either in the same cap or in the region bounded by trisectrix for $(\omega_{1}, AA'')$ and $A'T_{2}$, or on $T_{2}B$-ray. Location of $L$ implies that $K$ must be either in the $A$-vertex cap or in the region bounded by trisectrix for $(\omega_{1}, AA')$ and $AT_{5}$. In all cases, $NK$ necessarily cuts at least one of the disks. Thus, $M$ can not be in the $A'$-vertex cap of $A'BCA''$. Similarly, it can not be in the $A''$-vertex cap of $A'BCA''$. 
\item[(c)] Assume $M\in int(\triangle AA'A'')$. Then $N$ is located in one of $A'$-vertex cap or the interior of the region bounded by trisectrix for $(\omega_{1}, AA'')$ and $A'T_{2}$. The former is not possible by (b). $L$ is located in one of $A''$-vertex cap or the interior of the region bounded by the trisectrix for $(\omega_{1}, AA')$ and $A''T_{5}$. Again the former is not possible. Position of $N$ implies $K$ is either in $B''$-vertex cap, or above $B''B'$-line. The former case is not possible by (b). Similarly, position of $L$ implies $K$ is located below $C'C''$-line. Let $Q$ be the intersection of $B'B''$-line and $C'C''$-line. Our last two conclusions imply that $K$-vertex cap of $MNKL$ contains $\triangle B'QC''$, whose area is \[\|\triangle B'QC''\|=\frac{4+\sqrt{12}}{\sqrt{3}}. \tag{5}\] 

Notice that $\|MNKL\|>\|H\|+\|\triangle B'QC''\|$. On the other hand \[\|A'BCA''\|=\frac{11+6\sqrt{3}}{\sqrt{3}}. \tag{6}\]
By using $(4)$, $(5)$, $(6)$, we know $\|MNKL\|-\|A'BCA''\|>\|H\|+\|\triangle B'QC''\|-\|A'BCA''\|>0$. Thus, MNKL is not the smallest.
\item[(d)] Assume $M\in (R_{1}A')$. Then $L$ is on the section strictly between $A''$ and $I_{3}$ of the trisectrix for $(\omega_{1}, AA')$, while $N$ is on the $A'B$-ray. The former conclusion implies $K$ is located above $BC$-line and below $I_{3}R_{4}$-line. Then $NK$ can only be tangent to $\omega_{2}$. This in turn implies $N\in (T_{2}B)$. Then distance from $N$ to a point on $T_{2}T_{3}$-arc of $\omega_{2}$ is much smaller than distance from $K$ to the same point. In particular the former distance is at most $|BT_{3}|=\sqrt{3}$, while the latter distance is greater than $|T_{3}T_{4}|=4$ (to see this just project $K$ to $BC$). Thus, such $MNKL$ can not be the smallest. Similarly, $M\in (R_{6}A'')$ gives a non-optimal solution.
\item[(e)] Assume $M\in (AR_{1})$. Then $L$ is in the interior the region bounded by the trisectrix for $(\omega_{3}, CC'')$. This implies $K\in int(\triangle CC'C'')$. We can repeat the argument in (c) to reach a contradiction. Similarly, $M\in (AR_{6})$, then we reach to a contradiction.
\item[(f)] Assume $M=R_{1}$. Then $L=I_{3}$, which implies $K=R_{4}$. These imply that $MNKL=R_{1}BR_{4}I_{3}$. Similarly, $M=R_{6}$ implies $MNKL=R_{6}I_{1}R_{3}C$.
\item[(g)] Assume $M=A$. Then $N$ is on $T_{2}B$-ray and $L$ is on $T_{5}C$-ray. Since both $NK$ and $KL$ are tangent to $H$, we must have either $L\in T_{5}C$ or $N\in T_{2}B$. First assume $L\in T_{5}C$. If $L\in (T_{5}C)$, then by arguments above we know that either $L=C'$ or $L=R_{5}$. In the first case we end up with $ABC''C'$ and in the second case with $AR_{2}I_{2}R_{5}$. If $L=C$, then $MNKL=AB''B'C$. Similarly $N\in(T_{2}B)$ implies $MNKL\in\{B'CAB'', I_{2}R_{5}AR_{2}, ABC''C'\}$.
\end{itemize}
Thus, we conclude that when $M$ is located in the $A$-vertex cap of $\triangle ABC$, $MNKL\in \{A'BCA'', B'CAB'', C'ABC'', I_{1}R_{3}CR_{6}, I_{2}R_{5}AR_{2}, I_{3}R_{1}BR_{4}\}$. A similar argument shows that $MNKL$ is one of these six polygons when $M$ is located in the $C$-vertex cap or $B$-vertex cap of $\triangle ABC$. This proves our claim.
	
Symmetries in our configuration imply $\|A'BCA''\|=\|B'CAB''\|=\|C'ABC''\|$ and $\|I_{1}R_{3}CR_{6}\|=\|I_{2}R_{5}AR_{2}\|=\|I_{3}R_{1}BR_{4}\|$. Then $\|A'BCA''\|=\frac{11+6\sqrt{3}}{\sqrt{3}}>12=\|P^{\ast}\|$. Notice that $\|I_{3}R_{1}BR_{4}\|=\|R_{1}P_{1}OI_{3}\|+\|P_{1}BP_{2}O\|+\|P_{2}CI_{3}O\|$. Since each of these three quadrilaterals contains a unit disk and is non-regular, each has area greater than $4$ by Theorem \ref{1-thm}. Then, $\|I_{3}R_{1}BR_{4}\|>12=\|P^{\ast}\|$.
\end{proof}
\noindent{\em Remarks}: By now we know that $n=3, k=4$ is the first case where the disks are packed non-efficiently inside the smallest containing polygon, when we order $(n,k)$ lexicographically.

\section{Final discussions}
\label{final}
Let us discuss some open problems. Because of Theorem \ref{non-wegner}, packing of $n$ disks in the smallest convex $k$-gon is not always the most efficient. However by Dowker inequality \cite{dowker}, at any fixed location of the disks, the area of the smallest containing $k$-gon tends rather fast to the area of their convex hull as $k\rightarrow\infty$. This is a supportive fact for the efficient packing to realise inside the smallest containing $k$-gon when $k$ is large. It is believed that efficient packing of $n\in\mathbb{N}$ unit disks is a Groemer packing \cite{boroczky-ruzsa}. We can then ask whether the packing of $n$ disks in the smallest convex $k$-gon is always a Groemer packing? 

The following assertion seems very plausible.
\begin{conjecture}
If $n=\frac{m(m+1)}{2}$, i.e. the $m'$th triangular number, then the smallest triangle containing $n$ unit disks is the equilateral triangle of side $2(m-1)+2\sqrt{3}$. 
\end{conjecture}
In Theorem \ref{1-thm} and Theorem \ref{6k-thm} (b) we proved it for $m=1,2,3$, but our approach stopped at $m=4$ where the Wegner packing of $10$ disks is not triangular; see Fig. \ref{wegner-fig}. One can show that the smallest triangle containing two unit disks is the isosceles right triangle with hypothenuses of length $6+4\sqrt{2}$, i.e. it is not equilateral. However, it is plausible that the smallest triangle containing $\frac{m(m+1)}{2}-1$ unit disks with $m>2$ is equilateral. These lead to a new version of the long standing Erd\H{o}s-Oler conjecture \cite{oler}.
\begin{conjecture}
If $n>3$ is a triangular number, then the smallest triangle containing $n$ disks is the same as that containing $n-1$ disks. 
\end{conjecture}
In Theorem \ref{6k-thm}, we showed that if $n$ is a centred hexagonal number, then the smallest containing hexagon is regular. Then, we can propose the following analogy of this conjecture. 
\begin{conjecture}
If $n\in\mathbb{N}$ is a centred hexagonal number, then the smallest hexagon containing $n$ disks is the same as that containing $n-1$ disks.
\end{conjecture}

The following assertion would simplify some of our proofs. Let $P$ be the smallest convex $k$-gon containing a convex disk $C$. Let $P^{1}$ be the convex polygon obtained after cutting the largest (area) triangles from each cap of $P$ so that $P^{1}$ still contains $C$. Is it true that minimality of $P$ implies minimality of $P^{1}$? The answer is positive if $C$ is a unit disk, negative if $C$ is the convex hull of two tangent unit disks and $k=3$, but positive in the latter case when $k=4$ (see Theorem \ref{6k-thm} (a)). What if $C$ is an ellipse? Moreover, does minimality of both $P$ and $P^{1}$ imply minimality of members of the sequence of obtained by the 'greedy cut' procedure described?

\bigskip\noindent{\bf Acknowledgements}: We are grateful to participants of 2018 Fall Meeting of the Mongolian Mathematical Society for useful discussions and to Bilguun Bolortuya for research assitance. Financial support from the Asia Research Centre in Mongolia (P2018-3567) is acknowledged.

\end{document}